\documentclass[a4paper,10pt]{article}
\usepackage{graphicx}
\usepackage{float}
\usepackage{amsthm}
\usepackage{amsfonts}
\usepackage[latin1]{inputenc}
\usepackage{indentfirst}
\usepackage{amsmath}
\usepackage{amssymb}
\usepackage{color}

\newtheorem{theo}{Theorem}

\newtheorem{lemma}{Lemma}

\newtheorem{ann}{Assumption}

\numberwithin{equation}{section}
\numberwithin{theo}{subsection}

\newcommand{\be} {\begin{equation}}
\newcommand{\ee} {\end{equation}}
\newcommand{\bea} {\begin{eqnarray}}
\newcommand{\eea} {\end{eqnarray}}
\newcommand{\Bea} {\begin{eqnarray*}}
\newcommand{\Eea} {\end{eqnarray*}}

\begin{document}
\title{Limit theorems for strongly and intermediately supercritical branching processes in random environment with linear fractional offspring distributions}
\author{Christian B\"oinghoff}
\maketitle \vspace{1.5cm}

\begin{abstract} In the present paper, we characterize the behavior of 
supercritical branching processes in random environment with linear fractional offspring distributions, 
conditioned on having small, but positive values at some large generation. As it has been noticed in previous works, there 
is a phase transition in the behavior of the process. Here, we examine the strongly and intermediately supercritical regimes 
The main result is a conditional limit theorem for the rescaled associated random walk in the intermediately case.  \end{abstract}

{\em AMS 2000 Subject Classification.} 60J80, 60K37, 92D25

{\em Key words and phrases.} branching process in random environment, supercritical, conditional limit theorem
\maketitle
\section{Introduction}
Branching processes in random environment (BPRE) are a stochastic model for the development of a population in discrete time. The model has first been introduced in 
\cite{at,sm}. In generalization to Galton-Watson processes, the reproductive success of all individuals of a generation is 
influenced by an environment which varies in an independent fashion from generation to generation. \\
As first noted in \cite{af_80,de}, there is a phase transition in the behavior of subcritical BPRE (see e.g. for an overview \cite{bgk} and for detailed results 
 \cite{agkv05,agkv2,gkv,va_04,dy04,dy05,dy07, BK13}). Only recently, there has been interest in a phase transition
in supercritical processes, conditioned on surviving and having small values at some large generation (see \cite{BB12,BB13, nak13}). 
For the scaling limit of supercritical branching diffusions, a phase transition has been noted in \cite{Hutz12}. \medskip\\
In \cite{BB12,Hutz12}, the terminology of strongly, intermediately and weakly supercritical BPRE has been introduced in analogy to subcritical BPRE. In the present paper, 
we focus on the phase transition from strongly to intermediately supercriticality and characterize these regimes with limit results. \\
Let us formally introduce a branching process in random environment $(Z_n)_{n\in\mathbb{N}_0}$. For this, let $Q$ be a random variable taking values in $\Delta$, 
the space of all probability distributions on $\mathbb{N}_0=\{0,1,2,3,\ldots\}$. An infinite sequence $\Pi=(Q_1,Q_2,\ldots)$ of i.i.d. 
copies of $Q$ is called a \textsl{random environment}. By $Q_{n}$, we denote the (random) offspring distribution of an individual at generation $n-1$. \\
Let $Z_n$ be the number of individuals in generation $n$. Then $Z_{n}$ is the sum of $Z_{n-1}$ independent random variables with distribution $Q_{n}$. 
A sequence of $\mathbb{N}_0$-valued random variables  $Z_{0},Z_{1},\ldots$ is then called a 
  \emph{branching process in the random environment} $\Pi$, if $Z_{0}$ is independent of $ \Pi$ and
 given $ \Pi$ the  process $Z=(Z_{0},Z_{1},\ldots)$ is a Markov chain with 
 \begin{equation}  \label{transition} 
     \mathcal{L} \big(Z_{n} \; \big| \; Z_{n-1}=z, \, \Pi 
    = 
      (q_{1},q_{2},\ldots) \big) \ = \ q_{n}^{*z} 
 \end{equation} 
 for every $n\in \mathbb{N}=\{1,2,\ldots\}$, $z \in \mathbb{N}_0$ and $q_{1},q_{2},\ldots \in \Delta$, where $q^{*z}$  is the $z$-fold convolution of the measure $q$. \\
By $\mathbb{P}$, we will denote the corresponding probability measure on the underlying probability space. We shorten  $Q(\{y\}), q(\{y\})$ to $Q(y),q(y)$ and write 
\[\mathbb{P}(\cdot \mid Z_0=z) =: \mathbb{P}_z(\cdot) \ .\]
For convenience, we write $\mathbb{P}(\cdot)$ instead of $\mathbb{P}_1(\cdot)$. 
Throughout the paper, we assume that the offspring distributions have the following form,
\[q\big(0\big)= a \ , \ q\big(k\big) = \frac{(1-a)(1-p)}{p}\ p^k \ , \ \text{for}\ k\geq 1\]
where $a\in [0,1)$ and $p\in(0,1)$ are two random parameters. Note that $a=1$ would imply that 
an individual becomes extinct with probability one. Thus we exclude this case.
This class of offspring distributions is often called \emph{linear fractional} as the generating functions have an 
explicit formula as a quotient of two linear functions. Also the concatenation of two linear fractional 
functions is again linear fractional and thus the generating function of $Z_n$ can be calculated explicitly in this case. 
An important tool in the analysis of BPRE is the associated random walk $S=(S_n)_{n\ge 0}$. This random walk has initial state 
$S_{0}=0$ and  increments   $X_n=S_n-S_{n-1}, \, n\ge 1$ defined by
\[ 
    X_{n} \ := \ \log m(Q_n),  
\] 
where 
\[ 
m(q)\ := \
\sum_{y=0}^{\infty} y \ q(y) 
\] 
is the mean of the offspring distribution $q \in \Delta$. The expectation of $Z_n$ can be expressed by $S_n$ by
\begin{equation*} 
  \mathbb{E}[Z_{n} \,| \,  \Pi \,]   \ = \ 
\prod_{k=1}^n m(Q_k)\ = \ \exp(S_n) \quad 
\mathbb{P}\text{--a.s.} 
\end{equation*} 
Averaging over the environment gives 
\begin{align}   
  \mathbb{E}[Z_{n} ]   \ = \ 
 \mathbb{E}[ m(Q) ]^n.  \label{expect2} 
\end{align} 
A well-known estimate following from this by Markov inequality is
\begin{eqnarray} 
   \lefteqn{ \mathbb{P} (Z_{n} > 0 \,| \,  \Pi ) \ =  \ 
\min_{0\le k \leq n} 
       \mathbb{P} (Z_{k} > 0 \, | \,  \Pi )}     \nonumber\\ 
    &\leq & 
  \min_{0\le k \leq n}   \mathbb{E}[ Z_{k} \,| \,  \Pi ]  \ = \ 
\exp\big(\min_{0\le k \leq n} S_{k}\big)  \quad 
\mathbb{P}\text{--a.s.}       \label{gleichung1} 
\end{eqnarray} 
Here, we focus on supercritical BPRE, i.e. the case of $\mathbb{E}[X]>0$. 

As it has been described in \cite{BB12} on the level of large deviations and for the most recent common ancestor, 
there is a phase transition in the supercritical regime. Our aim here is to describe the regimes of \emph{strongly} and \emph{intermediately} supercriticality 
more in detail. In these regimes, the event $\{Z_n=1\}$ is typically realized in a favorable environment, i.e. conditioned on $\{Z_n=1\}$, $S_n$ will still be large. 

Throughout the paper, we write $\{Z_\infty>0\}$ for the event $\{ Z_n>0 \ \forall n\in\mathbb{N}\}$. The paper is organized as follows: In Section 2, 
we state our main results. Section 3 deals with special properties of linear fractional offspring distributions. Section 4 recalls some properties of 
conditioned random walks whereas our results are proved in Sections 5 and 6.

\section{Results}
\subsection{The strongly supercritical case}
In this part, we assume that
\begin{align}\label{strong1}
0<\mathbb{E}\big[Xe^{-X}\big]<\infty\ , 
\end{align}
which we refer to as \emph{strongly supercritical}. Note that this condition implies $\mathbb{E}[e^{-X}]<\infty$. First, we introduce a change of measure. Let
 $\phi: \Delta\rightarrow \mathbb{R}$ be a bounded 
and measurable function. Then the measure $\mathbf{P}$ is defined by
\begin{align}\label{change1}
\mathbf{E}\big[\phi(Q)\big]:= \gamma^{-1} \mathbb{E}\big[e^{-X}\phi(Q)\big] \ ,
\end{align}
where $\gamma:=\mathbb{E}\big[e^{-X}\big]$. Under $\mathbf{P}$, $S$ is still a random walk with positive drift as (\ref{strong1}) implies
\[\mathbf{E}\big[X\big]=\gamma^{-1}\mathbb{E}\big[Xe^{-X}\big]>0\ . \]
Throughout this section, we assume that $\mathbf{E}\big[\log(1-Q(0))\big]>-\infty$. As it is proved in \cite{sm}[Theorem 3.1], this condition assures that $\mathbf{E}[X]>0$ indeed implies that the process survives with a positive probability, i.e.
\[\lim_{n\rightarrow\infty} \mathbf{P}(Z_n>0)=\mathbf{P}(Z_\infty>0)>0 \ .\]
In the intermediately supercritical case, we require a slightly different condition.
Our first result describes the asymptotics of having exactly one individual at some large generation. 
\begin{theo}\label{thstrong1}
Under (\ref{strong1}), there is a constant $\vartheta>0$ such that 
\[\mathbb{P}(Z_n=1) \sim \vartheta \ \gamma^n \ .\]
\end{theo}
The next theorem describes the distribution of $Z_n$, conditioned on $Z_n$ being bounded.
\begin{theo}\label{theostrong23}
Assume (\ref{strong1}). As $n\rightarrow\infty$ and for every $c\in\mathbb{N}$ and $1\leq k\leq c$,
\[\mathbb{P}(Z_n=k\mid 1\leq Z_n\leq c) \rightarrow \tfrac{1}{c} \ , \]
i.e. the limiting distribution is uniform on $\{1,\ldots,c\}$. 
\end{theo}
The fact that this limiting distribution is uniform seems rather linked to linear fractional offspring distributions whereas we suspect that Theorem \ref{thstrong1} also holds for more general offspring distributions. \\
Our next theorem essentially says that conditioned on $Z_{n}=1$, the process is of constant order at all times. 
\begin{theo}\label{thstrong2}
Under Assumption (\ref{strong1}), there is a probability distribution $r=(r_z)_{z\in\mathbb{N}}$ such that for all $t\in (0,1)$
\[\lim_{n\rightarrow\infty} \mathbb{P}(Z_{\lfloor nt\rfloor} =z\mid Z_n=1) = r_z \ . \]
Moreover, $r_z$ does not depend on $t$ and is given by
\[r_z=z\mathbf{E}\big[\mathbb{P}(Z_{\infty} >0\mid\Pi)^2\mathbb{P}(Z_{\infty} =0\mid\Pi)^{z-1}\big]  \ .\]
\end{theo}
Let us briefly explain the intuition behind these results. Given $S_n\approx 0$ and $\min_{0\leq k\leq n} S_k\geq 0$, the process would have a large probability of 
surviving and being small at generation $n$ (see \cite{BB12}) and the event would be realized by \emph{environmental stochasticity}. However, in the 
strongly supercritical case, such an environment has very small probability. Conditioned on $\{Z_n=1\}$, $S_n$ has still a positive drift and thus $S_n$ will be large. Thus 
the event $\{Z_n=1\}$ is here typically realized by \textsl{demographic stochasticity}. Growing within a favorable environment and 
then becoming small again would require an exponential number of independent subtrees becoming extinct. This has very small probability. Thus the conditioned process 
typically stays small at all generations as it is stated in Theorem \ref{thstrong2}.\\ 
This theorem also explains a result for the most recent common ancestor in \cite{BB12}[Corollary 2.3]. If the process is small at all times, then the
 most recent common ancestor of the population will be close to the final generations. We suspect that qualitatively, these results will also be true for more general offspring distributions. 

\subsection{The intermediately supercritical case}
In this section, we assume that
\begin{align}\label{inter1}
\mathbb{E}\big[Xe^{-X}\big]=0\ . 
\end{align}
In \cite{BB12}, this regime has been characterized as \emph{intermediately supercritical}. Here, we will prove conditional limit theorems describing 
this regime more in detail.\\
(\ref{inter1}) suggests a change of measure. Recall the definition of $\mathbf{P}$ from (\ref{change1}). Due to 
Assumption (\ref{inter1}), $S$ is now a recurrent random walk under $\mathbf{P}$, i.e. 
\[\mathbf{E}[X]=0\ .\]
For our theorems, we require some regularity of the distribution of $X$. 
\begin{ann}\label{ass1}
We assume that with respect to $\mathbf{P}$, $X$ has finite variance, or more generally belongs to the domain of attraction of some strictly stable law $s$
 with index $\alpha\in (0,2]$. 
\end{ann}
As to the regularity of the offspring distribution, we require the following condition.
\begin{ann}\label{ass2}
We assume that there is an $\varepsilon>0$ such that
\[\mathbf{E}\big[|\log(1-Q(0))|^{\alpha+\varepsilon}\big]<\infty\ .\]
\end{ann}
Our first theorem describes the asymptotics of the probability of having exactly one individual in some large generation $n$. 
\begin{theo}\label{theo1}
Assume (\ref{inter1}). Then under Assumptions \ref{ass1} and \ref{ass2}, there is a positive and finite constant $\theta$ such that as $n\rightarrow\infty$,
\[\mathbb{P}(Z_n=1) \ \sim\ \theta \ \mathbb{E}\big[e^{-X}\big]^n \ \mathbf{P}\big(\min_{0\leq k\leq n} S_k\geq 0\big) \ . \]
\end{theo}
For a formula for $\theta$, see Section \ref{secint}. 
From \cite{agkv05}[Lemma 2.1], it results that there is a slowly varying sequence $l(n)$ such that
\[\mathbb{P}(Z_n=1) \ \sim\ \theta \ \gamma^n \ l(n) \ n^{-(1-\rho)} \ ,\]
where $\rho=s(\mathbb{R}^+)$ and $s$ is the limiting stable law from Assumption \ref{ass1}. \\
Our next theorem describes the environment, conditioned on $\{Z_n=1\}$. For this, define 
\[ S^n=\big(S^{n}_t\big)_{0\leq t\leq 1} := \tfrac{l(n)}{ n^{1/\alpha}} \big(S_{\lfloor nt\rfloor}\big)_{0\leq t\leq 1}\ .\] 
Then it holds that:
\begin{theo}\label{theo2}
Assume (\ref{inter1}). Under Assumptions  \ref{ass1} and \ref{ass2},
\begin{align*}
 \mathcal{L}\big(S^n\mid Z_n=1\big) \longrightarrow (L^+_t)_{t\in[0,1]} \ ,
\end{align*}
where $L^+$ is the meander of a strictly stable L\'evy process.
\end{theo}
Essentially $L^+$ is a strictly stable L\'evy process, conditioned on staying positive on $(0,1]$. For more details see e.g. \cite{agkv05,do,du}.
We can also specify the distribution of $Z_n$, given $1\leq Z_n\leq c$ as in the strongly supercritical case.
\begin{theo}\label{theo23}
Assume (\ref{inter1}) and Assumptions \ref{ass1} and \ref{ass2}. As $n\rightarrow\infty$ and for every $c\in\mathbb{N}$ and $1\leq k\leq c$,
\[\mathbb{P}(Z_n=k\mid 1\leq Z_n\leq c) \rightarrow \tfrac{1}{c} \ , \]
i.e. the limiting distribution is uniform on $\{1,\ldots,c\}$. 
\end{theo}
For our next theorem, we require the successive global minima. Define the time of the first minimum up to generation $n$ as
\[\tau_n:= \min\big\{0\leq k\leq n\mid S_k=\min\{0, S_1,\ldots, S_n\}\big\}\ .\]
By 
\begin{align}
 \tau_{k,n}:= \min\big\{k\leq j\leq n\mid S_j=\min\{S_k,\ldots, S_n\}\big\}\ ,
\end{align}
we denote the time of the first minimum between generations $k$ and $n$. Our next theorem proves that the conditioned BPRE is small 
in those minima.
\begin{theo}\label{theo3}
Assume (\ref{inter1}) and Assumptions \ref{ass1} and \ref{ass2}. Let $t\in (0,1)$. Then
\[\lim_{n\rightarrow\infty} \mathbb{P}(Z_{\tau_{\lfloor nt\rfloor,n}}=z\mid Z_n=1) = q(z) \ ,\]
where
\[q(z):=z\ \mathbf{E}\big[\mathbb{P}(Z_{\infty} >0\mid\Pi)^2\mathbb{P}(Z_{\infty} =0\mid\Pi)^{z-1}\big] \ \]
is a probability distribution on $\mathbb{N}$.
\end{theo}
Let us briefly explain these results. Conditioned on $\{Z_n=1\}$, $S_n$ does not have a positive drift but converges -properly rescaled- to a L\'evy meander. 
Thus the conditioned environment is less favorable than in the strongly supercritical case. In the minima of the L\'evy meander, the process
 has to be small. Otherwise, many individuals would exhibit an environment favorable for growth making the event $\{Z_n=1\}$ too improbable. In 
the excursions of $S$ in between these minima, we expect (but did not prove here) that $Z$ may grow and have exponentially large values. Thus \textsl{environmental} and 
\textsl{demographic} stochasticity have equal importance in this case.

mid\section{The linear fractional offspring distribution}
\textbf{Remark.} Throughout the paper, we only consider generalized geometric offspring distributions. Only for these distributions, an explicit formula for $\mathbb{P}(Z_n=1\mid \Pi)$ exists (except for some special related cases). As we condition on this probability, a useful formula is necessary in all our proofs. In the subcritical cases, one typically conditions on $\mathbb{P}(Z_n>0\mid \Pi)$. For this probability, a useful formula is known also for general offspring distributions (see e.g. \cite{ag}), involving the second factorial moments and not depending on the fine structure of the offspring distributions. Thus in contrast to the subcritical cases, we can only proof our results here for linear fractional offspring distributions. As in contrast to $\mathbb{P}(Z_n>0\mid\Pi)$, $\mathbb{P}(Z_n=1\mid \Pi)$ seems to depend on the fine structure of the offspring distributions (see \cite{BB12}), generalizing the results seems to be difficult. However, the parallels to the subcritical case indicate that the theorems might be qualitively true for more general offspring distributions (aside from the explicit limiting distributions proved here).\qed \medskip\\ 
Now, we present details on generalized geometric offspring distributions. Let $\xi$ be a random variable on $\mathbb{N}_0$ with distribution $q$ and let
\[q(0)=a \ , \ q(k)=\frac{(1-a)(1-p)}{p}  \ p^k \ ,\]
where $a\in[0,1)$ and $p\in(0,1)$. The expectation of $\xi$ is
\begin{align*}
 m_\xi:= \sum_{k=1}^\infty k\ q(k)= \sum_{k=1}^\infty (1-a)(1-p) \ k  \ p^{k-1} = (1-a)(1-p)\cdot \frac{1}{(1-p)^2} = \frac{1-a}{1-p} \ .
\end{align*}
Inserting this, we get the formula
\begin{align*}
 q(k)=(1-a)(1-p) p^{k-1} = \frac{1}{m_\xi} (1-a)^2 p^{k-1} = \frac{(1-q(0))^2}{m_\xi} p^{k-1} \ ,
\end{align*}
i.e. uniformly in $k$,
\[0\leq m_\xi\cdot q(k) \leq 1\ .\]
This fact will be used later. We can also rewrite the probability weights as
\begin{align}\label{geo22}
 q(k) =q(1) p^{k-1}\ .
\end{align}
Let us now turn to the concatenation of linear fractional generating functions. As it is proved e.g. in \cite{kozlov06}[p. 156, Equation (6)], if the offspring 
distributions are linear fractional, then also $Z_n$, conditioned on $\Pi$ has a linear fractional offspring distribution given by
\begin{align*}
 \mathbb{P}(Z_n=z\mid\Pi) = \frac{e^{-S_n}}{\big(e^{-S_n}+\sum_{k=0}^{n-1} 
\eta_{k+1}e^{-S_k}\big)^2} \Big(\frac{\sum_{k=0}^{n-1} \eta_{k+1}e^{-S_k}}{e^{-S_n}+\sum_{k=0}^{n-1} \eta_{k+1}e^{-S_k}}\Big)^{z-1}\ \text{a.s.} ,
\end{align*}
 where (see \cite{kozlov06}[p. 155]) $\eta_k:=\eta_{Q_k}$ and
\begin{align*}
 \eta_q= \frac{\sum_{j=0}^\infty j (j-1) q(j)}{2m_{q}^2} = \frac{\frac{(1-a)2p}{(1-p)^2}}{2m_{q}^2}=\frac{p}{1-a}
\end{align*}
are the standardized second factorial moments of $Q_k$. Note that by summing over $z\in\mathbb{N}$, we get that a.s. 
\begin{align*}
 \mathbb{P}(Z_n>0\mid \Pi)=\frac{1}{e^{-S_n}+\sum_{k=0}^{n-1} 
\eta_{k+1}e^{-S_k}}\ .
\end{align*}
Defining 
\begin{align}\label{eqdefh}
 H_n:= \frac{\sum_{k=0}^{n-1} \eta_{k+1}e^{-S_k}}{e^{-S_n}+\sum_{k=0}^{n-1} \eta_{k+1}e^{-S_k}}
\end{align}
we get the convenient formula
\begin{align}\label{geo2}
 \mathbb{P}(Z_n=z\mid\Pi)= e^{-S_n} \mathbb{P}(Z_n>0\mid\Pi)^2 H_n^{z-1} \quad \text{a.s.} 
\end{align}
In particular, we will use (as already proved in \cite{BB12})
\begin{align}\label{surv1} 
\mathbb{P}(Z_n=1\mid\Pi) = e^{-S_n}\ \mathbb{P}(Z_n>0\mid\Pi)^2 \ \text{a.s.} 
\end{align}
By (\ref{eqdefh}), if $S_n\rightarrow\infty$ a.s. as $n\rightarrow\infty$, then also $H_n\rightarrow 1$ a.s.  \medskip\\
For our proofs, we require the generating function of a distribution $q$,
\[f(s):=\sum_{k=0}^\infty s^k \ q(k) \quad , s\in[0,1] \ .\]
Let $f_k$ be the generating function of $q_k$ and set
\[f_{0,n}(s):=\mathbb{E}\big[s^{Z_n} \mid Z_0=1\big] \ .\]
As it is well-known, $f_{0,n}$ is the concatenation of the generating functions of each generation (see e.g. \cite{agkv05}[Equation (3.2)]) , i.e. 
\[f_{0,n}(s)= f_1(f_2(\ldots (f_n(s))\ldots )) \ .\]
More generally, let for $k\leq n$
\[f_{k,n}(s):=\mathbb{E}\big[s^{Z_n} \mid Z_k=1\big] \ \]
 resp.
\[f_{k,n}(s)= f_{k+1}(f_{k+2}(\ldots (f_n(s))\ldots )) \ .\]
Thus we get the formula
\begin{align}\label{eq_gen1}
 \mathbb{P}(Z_n>0\mid\Pi)=1- f_{0,n}(0) \quad \text{a.s.} 
\end{align}
\textbf{Remark.} In \cite{agkv05}[Assumption B2], it is required that for some $\varepsilon>0$
\[\mathbf{E}\big[\big(\log^+ (\eta)\big)^{\alpha+\varepsilon}\big]<\infty\ ,\]
where $\log^+(x):=\log(\max(x,1))$. In our context, inserting the formula for $\eta$, this means
\[\mathbf{E}\big[\big(\log^+ \big(\frac{p}{1-a}\big)\big)^{\alpha+\varepsilon}\big]<\infty \ .\]
As $p<1$ and $\alpha\in (0,2]$ and $a=Q(0)$, this condition is implied by
\[\mathbf{E}\big[\big(\log(1-Q(0))\big)^{\alpha+\varepsilon}\big]<\infty\ .\]

\section{Properties of random walks}
Define for $n\geq 1$
\begin{align*}
 L_n:=\min\{S_1,\ldots, S_n\} \quad , \quad M_n:=\max\{S_1,\ldots,S_n\} 
\end{align*}
and set $L_0=M_0=0$. Recall the definition of $\tau_n$, 
\begin{align*}
 \tau_n= \min\{ 0\leq k\leq n \mid S_k=\min(0;L_n)\} \ . 
\end{align*}
Next, we require the renewal function $u: \mathbb{R}\rightarrow\mathbb{R}$ defined by
\begin{align*}
 u(x)&=1+\sum_{k=1}^\infty \mathbf{P}(-S_k\leq x, M_k<0) \quad , x\geq 0 \ ,\\
 u(x)&=0 \ , x<0\ .
\end{align*}
For more details, see e.g. \cite{agkv05}. It is well-known that $u(0)=1$. Let  
\[\mathbf{P}_x\big((S_1, \ldots,S_n) \in \cdot\big)=\mathbf{P}\big((S_1, \ldots,S_n) \in \cdot\mid S_0=x\big)\ .\]
Here, we require that under Assumption \ref{ass1}, for every $x\geq 0$ and as 
$n\rightarrow\infty$
\begin{align}\label{eq_min}
 \mathbf{P}_x(L_n\geq -x) \sim u(x)\ n^{-(1-\rho)}\ l(n) \ ,
\end{align}
where the sequence $l(n)$ is slowly varying at infinity (see \cite{agkv05}[Lemma 2.1]).\medskip\\
Furthermore, we require the $h$-transform describing the random walk conditioned on never entering $(-\infty,0)$ where we will denote the corresponding measure by $\mathbf{P}^+$. More precisely, for every oscillating random walk 
the renewal function $u$ has the property
\[u(x) = \mathbf{E}\big[u(x+X);x+X\geq 0\big] \ , \ x\geq 0 \ .\]
This martingale property allows to define the measure $\mathbf{P}^+$. The construction of this measure is described in detail in e.g. \cite{agkv05, do}. Here, we only briefly recall the definition. \\
Let $\mathcal{F}_n=\sigma\big(Q_1,\ldots, Q_n, Z_0,\ldots,Z_n\big)$ be the $\sigma$-algebra generated by the branching process and its environment 
up to generation $n$. Let $R_n$ be a bounded random variable adapted to the filtration $(\mathcal{F}_n)_{n\in\mathbb{N}}$. Then $\mathbf{P}^+$ is defined by
\begin{align*}
 \mathbf{E}_x^+\big[R_n\big]= \frac{1}{u(x)} \mathbf{E}_x\big[R_nu(S_n); L_n\geq 0 \big] \ , \ x\geq 0 \ .
\end{align*}
We will later use the following result from \cite{agkv05}[Lemma 2.5].
\begin{lemma}\label{lem25}
Let $U_n$ be a uniformly bounded sequence of random variables adapted to the filtration $(\mathcal{F}_n)_{n\in\mathbb{N}}$. If $U_n\rightarrow U_\infty$ $\mathbf{P}^+$-a.s. for some 
limiting random variable $U_\infty$, then as $n\rightarrow\infty$, it holds that
\[ \mathbf{E}\big[U_n \mid L_n \geq 0 \big]\rightarrow \mathbf{E}^+\big[U_\infty\big] \ .\]
\end{lemma}

\section{Proof of theorems in the strongly supercritical case}
We assume throughout this section that $\mathbb{E}\big[Xe^{-X}\big]>0$ and thus (using Definition (\ref{change1}))
\[\mathbf{E}[X]>0 \ .\]
\subsection{Proof of Theorem \ref{thstrong1}} Using (\ref{surv1}) and the change of measure (\ref{change1}), we get that
\begin{align*}
\mathbb{P}(Z_n=1) &= \mathbb{E}\big[e^{-S_n}\mathbb{P}(Z_n>0\mid \Pi)^2\big]\\
&=\gamma^n \mathbf{E}\big[\mathbb{P}(Z_n>0\mid\Pi)^2\big] \ .
\end{align*}
Under $\mathbf{P}$, $Z$ is still a supercritical branching process in random environment with linear fractional 
offspring distributions. As it is proved in \cite{sm}[Theorem 3.1], under the condition $\mathbf{E}\big[\log(1-Q(0))\big]>-\infty$,
\[\lim_{n\rightarrow\infty} \mathbf{P}(Z_n>0)=\mathbf{P}(Z_\infty>0)>0 \ .\]
As the survival probability is monotonically decreasing, the limit above exists and we conclude  
\[ \mathbf{P}\big(\mathbb{P}(Z_\infty>0\mid\Pi)>0\big)>0 \ .\]
By monotonicity, we may interchange the limit and the integration and get that 
\begin{align*}
\lim_{n\rightarrow\infty} \gamma^{-n} \mathbb{P}(Z_n=1) &=\mathbf{E}\big[\mathbb{P}(Z_\infty>0\mid\Pi)^2\big]=:\vartheta>0 \ .
\end{align*}
This proves Theorem \ref{thstrong1}. \qed
\subsection{Proof of Theorem \ref{theostrong23}}
Using (\ref{geo2}), we have a.s. 
\begin{align*}
 \mathbb{P}(Z_n=z\mid \Pi)= e^{-S_n}\mathbb{P}(Z_n>0\mid\Pi)^2 H_n^{z-1} \ ,
\end{align*}
where
\begin{align*}
 H_n= \frac{\sum_{k=0}^{n-1} \eta_{k+1}e^{-S_k}}{e^{-S_n}+\sum_{k=0}^{n-1} \eta_{k+1}e^{-S_k}} \ .
\end{align*}
Using the change of measure (\ref{change1}) yields
\begin{align}
 \mathbb{P}(Z_n=z)&= \mathbb{E}\big[\mathbb{P}(Z_n=z\mid \Pi)\big]\nonumber\\
&= \gamma^n \mathbf{E}\big[\mathbb{P}(Z_n>0\mid \Pi)^2 H_n^{z-1}\big]\ . \label{eq_st10}
\end{align}
Under $\mathbf{P}$, $S_n\rightarrow \infty$ a.s. and therefore $H_n\rightarrow 1$ a.s. As $H_n$ is bounded by 1, we may use the dominated convergence theorem to conclude that
\begin{align}
 \lim_{n\rightarrow\infty} \gamma^{-n} \mathbb{P}(Z_n=z)= \mathbf{E}\big[\mathbb{P}(Z_\infty>0\mid \Pi)^2\big] \label{eq_stlim}
\end{align}
for every $z\in\mathbb{N}$. Inserting this yields for every $c\in\mathbb{N}$ and $1\leq z\leq c$
\begin{align*}
 \mathbb{P}(Z_n=z\mid 1\leq Z_n\leq c) &= \frac{\mathbb{P}(Z_n=z) }{\sum_{j=1}^c\mathbb{P}(Z_n=j) }=\frac{\gamma^{-n}\mathbb{P}(Z_n=z) }{\sum_{j=1}^c\gamma^{-n}\mathbb{P}(Z_n=j) } \ .
\end{align*}
Taking the limit $n\rightarrow\infty$ and using (\ref{eq_stlim}) yields the theorem. \qed
\subsection{Proof of Theorem \ref{thstrong2}} 
Let $z\in\mathbb{N}$ and $t\in (0,1)$ be fixed. Then using the branching property, independence of the environment and (\ref{surv1}) yield 
\begin{align*}
\mathbb{P}(Z_{\lfloor nt\rfloor} = z, Z_n=1) &=\mathbb{E}\big[\mathbb{P}(Z_{\lfloor nt\rfloor} = z, Z_n=1  \mid\Pi) \big]= \mathbb{P}(Z_{\lfloor nt\rfloor} = z) \mathbb{P}_z(Z_{n-\lfloor nt\rfloor}=1)\ .
\end{align*}
Conditioning on the environment and using (\ref{eq_stlim}), we get that as $n\rightarrow\infty$
\[\mathbb{P}(Z_{\lfloor nt\rfloor} = z)= \gamma^{\lfloor nt\rfloor}\big(\mathbf{E}\big[\mathbb{P}(Z_{\infty} >0|\Pi)^2\big]+o(1)\big)\ . \]
Starting from $z$-many individuals, $\{Z_n=1\}$ implies that $z-1$ of the $z$ subtrees must become extinct before time $n$. As conditioned on $\Pi$, all 
subtrees are independent and inserting (\ref{surv1}), we get that 
\begin{align*}
 \mathbb{P}_z&(Z_{n-\lfloor nt\rfloor}=1)=z\ \mathbb{E}\big[\mathbb{P}(Z_{n-\lfloor nt\rfloor} = 1\mid\Pi)\mathbb{P}(Z_{n-\lfloor nt\rfloor} =0\mid\Pi)^{z-1}\big]\\
&=z\ \mathbb{E}\big[e^{-S_{n-\lfloor nt\rfloor}} \mathbb{P}(Z_{n-\lfloor nt\rfloor} >0\mid\Pi)^2\mathbb{P}(Z_{n-\lfloor nt\rfloor} =0\mid\Pi)^{z-1}\big]\\
&= z\ \gamma^{n-\lfloor nt\rfloor} \mathbf{E}\big[\mathbb{P}(Z_{n-\lfloor nt\rfloor} >0\mid\Pi)^2\mathbb{P}(Z_{n-\lfloor nt\rfloor} =0\mid\Pi)^{z-1}\big] \ .
\end{align*}
By the dominated convergence theorem and monotonicity and as $\mathbf{P}(Z_{\infty}>0)>0$, 
\begin{align*}\lim_{n\rightarrow\infty}&\mathbf{E}\big[\mathbb{P}(Z_{n-\lfloor nt\rfloor} >0\mid\Pi)^2\mathbb{P}(Z_{n-\lfloor nt\rfloor} =0\mid\Pi)^{z-1}\big] \\
& = \mathbf{E}\big[\mathbb{P}(Z_{\infty} >0\mid\Pi)^2\mathbb{P}(Z_{\infty} =0\mid\Pi)^{z-1}\big]>0\ .
\end{align*}
Altogether, using Theorem \ref{thstrong1}, we get for $n\rightarrow\infty$
\begin{align*}
\mathbb{P}&(Z_{\lfloor nt\rfloor} = z\mid Z_n=1) =\frac{\mathbb{P}(Z_{\lfloor nt\rfloor} = z, Z_n=1)}{\mathbb{P}(Z_n=1)} \\
&= \frac{\mathbb{P}(Z_{\lfloor nt\rfloor} = z) \mathbb{P}_z(Z_{n-\lfloor nt\rfloor}=1)}{\gamma^n (\theta+o(1))} \\
&= \frac{\gamma^{\lfloor nt\rfloor} \big(\mathbf{E}\big[\mathbb{P}(Z_\infty>0\mid \Pi)^2\big]+o(1)\big) z \gamma^{n-\lfloor nt\rfloor} \big(\mathbf{E}\big[\mathbb{P}(Z_{\infty} >0\mid\Pi)^2\mathbb{P}(Z_{\infty} =0\mid\Pi)^{z-1}\big]+o(1)\big)}{\gamma^n (\vartheta+o(1))}\ .
\end{align*}
Recall that $\vartheta= \mathbf{E}\big[\mathbb{P}(Z_\infty>0\mid\Pi)^2\big]$.  Taking the limit $n\rightarrow\infty$ yields
\begin{align}\label{eq_dis1}
\lim_{n\rightarrow\infty} \mathbb{P}(Z_{\lfloor nt\rfloor} = z\mid Z_n=1) =
z\ \mathbf{E}\big[\mathbb{P}(Z_{\infty} >0\mid\Pi)^2\mathbb{P}(Z_{\infty} =0\mid\Pi)^{z-1}\big]=:r_z\ . 
\end{align}
In particular, this distribution does not depend on $t$. Summing over $z\in\mathbb{N}$, as for $x\in [0,1)$,
\[\sum_{z=0}^\infty z x^{z-1}=\frac{1}{(1-x)^2}\, \]
 and $\mathbb{P}(Z_\infty=0\mid\Pi)=1-\mathbb{P}(Z_\infty>0\mid\Pi)$ we see that (\ref{eq_dis1}) is indeed a probability distribution on $\mathbb{N}$. \qed\medskip\\

\section{Proof of theorems in the intermediately subcritical case}\label{secint}
Define for $0\leq k\leq n$
\begin{align}\label{min1}L_{k,n} = \min_{0\leq j\leq n-k} (S_{k+j}-S_k) \ . \end{align}
Note that with this definition
\begin{align}\label{eq_tn} \{\tau_n=k\}=\{\tau_k=k,L_{k,n}\geq 0\}\ .\end{align}
Due to independence of the increments, $L_{k,n}$ and $L_{n-k}$ have the same distribution. \\
First, we show that $Z_n$ being small implies that with a high probability, the associated random walk attains its minimum at some early generation.
\begin{lemma}\label{lemhelp}
For every $\varepsilon>0$ and $z\in\mathbb{N}$ there is an $m=m(\varepsilon)\in \mathbb{N}$ such that
\[ \mathbb{P}(Z_n=z, \tau_n>m)<\varepsilon \gamma^n \mathbf{P}(L_n\geq 0) \]
\end{lemma}
\textsl{Proof.} Decomposing at the minimum and using (\ref{geo2}) with $H_n\leq 1$ yields
\begin{align*}
  \mathbb{P}(Z_n=z, \tau_n>m) &=\sum_{k=m+1}^n \mathbb{E}\big[\mathbb{P}(Z_n=z\mid\Pi); \tau_n=k\big]\\
 &\leq\sum_{k=m+1}^n \mathbb{E}\big[\mathbb{P}(Z_n=1\mid\Pi); \tau_n=k\big]\\
&=\sum_{k=m+1}^n \mathbb{E}\big[e^{-S_n} \mathbb{P}(Z_n>0\mid\Pi)^2 ; \tau_n=k\big]\ .
\end{align*}
Next using the standard estimate (\ref{gleichung1}), (\ref{eq_tn}) and the change of measure, we get that
\begin{align*}
 \mathbb{P}(Z_n=z, \tau_n>m)&\leq \sum_{k=m+1}^n \mathbb{E}\big[e^{-S_n+2L_n}; \tau_k=k, L_{k,n}\geq 0 \big] \\
&= \gamma^n \sum_{k=m+1}^n \mathbf{E}\big[e^{L_n}; \tau_k=k, L_{k,n}\geq 0 \big] \\
&= \gamma^n \sum_{k=m+1}^n \mathbf{E}\big[e^{S_k}; \tau_k=k\big] \mathbf{P}(L_{n-k}\geq 0)\ .
\end{align*}
Using \cite{agkv05}[Lemma 2.2] with $u(x)=e^{-x}$, we get for $m\in\mathbb{N}$ large enough
\begin{align*}
 \sum_{k=m+1}^n \mathbf{E}\big[e^{S_k}; \tau_k=k\big] \mathbf{P}(L_{n-k}\geq 0)<\varepsilon \ \mathbf{P}(L_n\geq 0) \ .
\end{align*}
This yields the claim. \qed \medskip\\
In particular, we have proved that for every $\varepsilon>0$, if $m$ is large enough
\begin{align}\label{lem11}
 \limsup_{n\rightarrow\infty} \mathbf{P}(L_n\geq 0)^{-1} \sum_{k=m+1}^n \mathbf{E}\big[\mathbb{P}(Z_k>0\mid\Pi)^2; \tau_n=k\big]<\varepsilon \ .
\end{align}

\subsection{Proof of Theorem \ref{theo1}}
In the following proofs, we will require the shift of the environment. Let for $\Pi=(Q_1,Q_2,\ldots)$ and 
$k\in\mathbb{N}$ denote the environment shifted by $k$ generations by
\[\theta_k\circ \Pi=(Q_{k+1},Q_{k+2},\ldots)\ .\]
The main idea of the following lemma is to use results for critical BPRE, i.e. \cite{agkv05}[Lemma 4.1].
\begin{lemma}\label{lem10}
  Under Assumption \ref{ass1}, there is a positive and finite constant $\theta$ such that as $n\rightarrow\infty$
\[\mathbf{E}[\mathbb{P}(Z_n>0\mid\Pi)^2] \sim \ \theta\  \mathbf{P}(L_n\geq 0) \ .\]
\end{lemma}
\begin{proof}
Following \cite{agkv05}, we will first define a suitable sequence of uniformly bounded random variables which satisfy the conditions of \cite{agkv05}[Lemma 4.1]. In the second step, we will apply \cite{agkv05}[Lemma 4.1] and get the result of Lemma \ref{lem10} after a short calculation. \\
To avoid any confusion, we adopt the notation from \cite{agkv05}[Lemma 4.1]. Let $I_{Z_n>0}$ be the indicator function of the event $\{Z_n>0\}$. Then for  $n\in\mathbb{N}$ 
\[V_n:=\mathbb{P}(Z_n>0\mid\Pi) \cdot I_{Z_n>0}\quad \text{a.s.},\]
forms a uniformly bounded sequence of 
random variables $(V_n)_{n\in\mathbb{N}}$ adapted to the filtration $(\mathcal{F}_n)_{n\in\mathbb{N}}$. 
Next noting that $Z_n>0$ implies $Z_k>0$ and using (\ref{eq_gen1}), we get random variables adapted to $(\mathcal{F}_k)_{k\in\mathbb{N}}$
\begin{align*}
V_{k,n}& :=\mathbf{E}\big[\mathbb{P}(Z_n>0\mid\Pi) \cdot I_{Z_n>0} ; Z_{k}>0\mid L_{k,n}\geq 0,\mathcal{F}_k\big]\\
&=\mathbf{E}\big[\big(1-f_{0,k}\big(f_{k,n}(0)\big)\big) \cdot I_{Z_n>0}\mid L_{k,n}\geq 0,\mathcal{F}_k\big]\ \text{a.s.} 
\end{align*}
Both  $1-f_{k,n}(0)=\mathbb{P}(Z_n>0\mid\Pi, Z_k=1)$ a.s. and $I_{Z_n>0}$ (given $Z_k$) are bounded, nonnegative and nonincreasing in $n$ 
and thus converge $\mathbf{P}^+-$a.s. More precisely, given $\mathcal{F}_k$,
\[I_{Z_n>0}\stackrel{n\rightarrow\infty}{\longrightarrow} I_{Z_\infty>0}\quad \mathbf{P}^+-\text{a.s.}\]
 where $\mathbb{P}^+$ acts on $\theta_k\circ\Pi$. Moreover, given $\mathcal{F}_k$
\[\mathbb{P}(Z_n>0\mid\Pi) = 1-f_{0,k}\big(f_{k+1,n}(0)\big) \stackrel{n\rightarrow\infty}{\longrightarrow}  1-f_{0,k}\Big(P_\infty^k\Big) 
\quad \mathbf{P}^+-\text{a.s.}\]
where we defined
\begin{align}\label{eq_pk}P^k_\infty:=\mathbb{P}(Z_\infty=0\mid \theta_k\circ\Pi)\quad \text{a.s.}\end{align}
Thus the conditions of Lemma \ref{lem25} are met and as $n\rightarrow\infty$
\begin{align*}
 V_{k,n} \rightarrow V_{\infty}(Z_k,f_{0,k}) \quad \text{a.s.} \ , 
\end{align*}
where for $z\in\mathbb{N}_0$ and $g\in\mathcal{C}_b\big([0,1]\big)$
\begin{align*}
 V_{\infty}(z, g) &= \mathbf{E}^+_z\big[\big(1-g\big(P_\infty\big)\big)\cdot I_{Z_\infty>0}\big]\\
&= \mathbf{E}^+\big[\big(1-g\big(P_\infty\big)\big)\cdot \mathbb{E}_z[I_{Z_\infty>0}\mid\Pi]\big]\\
&=\mathbf{E}^+\big[\big(1-g\big(P_\infty\big)\big) \cdot\mathbb{P}_{z}(Z_\infty>0\mid\Pi)\big]\ .
\end{align*}
Next, recall that by \cite{agkv05}[Lemma 2.1], for fixed $k$ and as $n\rightarrow\infty$, $\mathbf{P}(L_n\geq 0)\sim \mathbf{P}(L_{k,n}\geq 0)$. Consequently, 
\begin{align*}
\mathbf{E}[V_n; Z_k>0,& L_{k,n}\geq 0\mid \mathcal{F}_k] = \mathbf{P}(L_{k,n}\geq 0) \mathbf{E}\big[\mathbb{P}(Z_n>0\mid\Pi) \cdot I_{Z_n>0} ; Z_{k}>0\mid L_{k,n}\geq 0,\mathcal{F}_k\big]\\
&= \mathbf{P}(L_n\geq 0) \big(V_{\infty}(Z_k,f_{0,k})+o(1)\big) \ \text{a.s.}
\end{align*}
and the conditions of  \cite{agkv05}[Lemma 4.1] are met for $m=0$. Thus by \cite{agkv05}[Lemma 4.1], we get that
\begin{align}\label{eq611}
\mathbf{E}[V_n; Z_{\tau_n}>0] = \mathbf{P}(L_n\geq 0) \Big(\sum_{k=0}^\infty \mathbf{E}[V_{\infty}(Z_k,f_{0,k}); \tau_k=k] +o(1)\Big) \ . 
\end{align}
Noting that $\{Z_{n}>0\}$ implies $\{Z_{\tau_n}>0\}$ and conditioning on the environment yields
\begin{align*}
\mathbf{E}[V_n; Z_{\tau_n}>0]&=\mathbf{E}\big[\mathbb{P}(Z_n>0\mid \Pi) \cdot I_{Z_n>0}; Z_{\tau_n}>0 \big]\\
&=\mathbf{E}\big[\mathbb{P}(Z_n>0\mid \Pi)\cdot I_{Z_n>0} \big] = \mathbf{E}\big[\mathbb{P}(Z_n>0\mid \Pi)^2 \big]  \ .
\end{align*}
Inserting this into  (\ref{eq611}), we get that
\[ \mathbf{E}\big[\mathbb{P}(Z_n>0\mid \Pi)^2 \big]=\mathbf{E}[V_n; Z_{\tau_n}>0]= \big(\theta +o(1)\big) \mathbf{P}(L_n\geq 0) ,\]
where 
\begin{align}\label{theta1}
\theta :=\sum_{k=0}^\infty \mathbf{E}[V_{\infty}(Z_k,f_{0,k}); \tau_k=k]\ .  
\end{align}
Clearly, $V_{\infty}(Z_k,f_{0,k})\leq \mathbf{P}^+_{Z_k}(Z_\infty>0)$, and thus by \cite{agkv05}[Equation (4.10)], the sum on the 
right-hand side is convergent.  As it is proved in \cite{agkv05}[Proposition 3.1], $P_\infty=\mathbb{P}_z(Z_\infty=0\mid\Pi)<1$ $\mathbf{P}^+$-a.s. 
for all $z\geq 1$. As for $s<1$ and $f_{0,k}(0)<1$ 
\[f_{0,k}(s) < f_{0,k}(1)= 1\]
this proves $\theta >0$. 
\end{proof}

$\qquad$\textsl{Proof of Theorem \ref{theo1}.} Using the change of measure and the explicit formula for $\mathbb{P}(Z_n=1\mid\Pi)$ in the case of linear fractional offspring distributions, we get that
\begin{align}
\mathbb{P}(Z_n=1)= \mathbb{E}\big[e^{-S_n} \mathbb{P}(Z_n>0\mid\Pi)^2\big] = \gamma^n \ \mathbf{E}\big[ \mathbb{P}(Z_n>0\mid\Pi)^2\big] \ .
\end{align}
The theorem now results from Lemma \ref{lem10}. \qed\medskip\\
There is another representation of $\theta$. Using $\mathbb{P}_z(Z_\infty>0\mid \Pi)=1-\mathbb{P}_1(Z_\infty=0\mid\Pi)^z$ yields
\begin{align*}
V_{\infty}(z,g) &=\mathbf{E}^+\big[\big(1-g\big(P_\infty\big)\big) \mathbb{P}_{z}(Z_\infty>0\mid\Pi)\big]\\
&= \mathbf{E}^+\big[\big(1-g\big(P_\infty\big)\big) \big(1-(P_\infty)^z\big)\big]\ .
\end{align*}
Taking into account the definition of generating functions and applying Fubini's theorem to interchange the expectations 
(note that $\mathbf{E}^+$ only acts on the shifted environment $\theta_k\circ\Pi$, i.e. $P^k_\infty$ whereas $\mathbf{E}$ only acts on $Z_k$ and $f_{0,k}$), we get the following 
representation of $\theta$.
\begin{align}
\mathbf{E}&\big[V_{\infty}(Z_k,f_{0,k}); \tau_k=k \big]= \mathbf{E}\Big[\mathbf{E}^+\Big[\big(1-f_{0,k}(P^k_\infty)\big)
\mathbf{E}^+\big[\big(1-(P^k_\infty)^{Z_k}\big)  \mid \theta_k\circ\Pi\big]\Big];\tau_k=k\Big]\nonumber\\
&=  \mathbf{E}\Big[\mathbf{E}^+\Big[\big(1-f_{0,k}(P^k_\infty)\big)\big(1-f_{0,k}(P^k_\infty)\big)\Big] ;\tau_k=k\Big]\nonumber\\
&= \mathbf{E}\big[\mathbf{E}^+\big[(1-f_{0,k}(P^k_\infty))^2\big]; \tau_k=k  \big]\ . \label{reptheta2}
\end{align}

\subsection{Proof of Theorem \ref{theo2}}
\begin{lemma}\label{lem12}
For every $\varepsilon>0$ there is an $m=m(\varepsilon)\in\mathbb{N}$ such that
\[\lim_{n\rightarrow\infty} \mathbf{P}(L_n\geq 0)^{-1} \mathbf{E}\big[\big|\mathbb{P}(Z_n>0\mid\Pi)^2-\mathbb{P}(Z_{\tau_n+m}>0\mid\Pi)^2\big|\big] <\varepsilon \ .\]
\end{lemma}
\begin{proof}
 We decompose according to $\tau_n$ and let $0\leq l\leq n$. Then
\begin{align}
\mathbf{E}\big[\big|&\mathbb{P}(Z_n>0\mid\Pi)^2-\mathbb{P}(Z_{\tau_n+m}>0\mid\Pi)^2\big|\big]\leq \sum_{k=0}^n \mathbf{E}\big[\big|\mathbb{P}(Z_n>0\mid\Pi)^2-\mathbb{P}(Z_{\tau_n+m}>0\mid\Pi)^2\big|;\tau_n=k\big]  \nonumber\\
&= \sum_{k=0}^{l } \mathbf{E}\big[\big|\mathbb{P}(Z_n>0\mid\Pi)^2-\mathbb{P}(Z_{k+m}>0\mid\Pi)^2\big|;\tau_n=k\big] \nonumber \\
&\quad +  \sum_{k=l+1}^n \mathbf{E}\big[\big|\mathbb{P}(Z_n>0\mid\Pi)^2-\mathbb{P}(Z_{k+m}>0\mid\Pi)^2\big|;\tau_n=k\big] \ .\label{eq11}
\end{align}
As to the second term, using the fact that $\mathbb{P}(Z_n>0\mid\Pi)$ is a.s. decreasing in $n$,  we get that 
\begin{align*}
\sum_{k=l+1}^n &\mathbf{E}\big[\big|\mathbb{P}(Z_n>0\mid\Pi)^2-\mathbb{P}(Z_{k+m}>0\mid\Pi)^2\big|;\tau_n=k\big]\\
&\leq  \sum_{k=l+1}^n \mathbf{E}\big[\mathbb{P}(Z_k>0\mid\Pi)^2;\tau_n=k\big] \ .
\end{align*}
Using (\ref{lem11}), if $l$ is chosen large enough, 
\begin{align}
\sum_{k=l+1}^n \mathbf{E}\big[\mathbb{P}(Z_k>0\mid\Pi)^2;\tau_n=k\big] \leq  \frac{\varepsilon}{2} \ \mathbf{P}(L_n\geq 0) \ .
\end{align}
For the first term in (\ref{eq11}), we condition on the environment up to generation $k$ and get that
\begin{align}
\sum_{k=0}^{l } &\mathbf{E}[\big|\mathbb{P}(Z_n>0\mid\Pi)^2-\mathbb{P}(Z_{\tau_n+m}>0\mid\Pi)^2\big|;\tau_n=k] \\
&\leq \sum_{k=0}^{l } \mathbf{E}\big[\big|\mathbb{P}(Z_n>0\mid\Pi)^2-\mathbb{P}(Z_{\tau_n+m}>0\mid\Pi)^2\big|;\tau_k=k, L_{k,n}\geq 0\big] \nonumber\\
&= \sum_{k=0}^{l }\mathbf{E}\big[\mathbf{E}\big[\big|\mathbb{P}(Z_n>0\mid\Pi)^2-\mathbb{P}(Z_{k+m}>0\mid\Pi)^2\big|; L_{k,n}\geq 0 \mid \mathcal{F}_k\big];\tau_k=k\big]\nonumber\\
&= \sum_{k=0}^{l} \mathbf{E}[\psi(f_{0,k},m,n-k) ;\tau_k=k] \mathbf{P}(L_{n-k}\geq 0) \ , \label{eq0705}
\end{align}
where for $g\in\mathcal{C}_b\big([0,1]\big)$, the space of bounded and continuous functions on $[0,1]$, we define 
\[\psi(g,m,n):= \mathbf{E}\big[\big|\big(1-g\big(f_{0,n}(0)\big)\big)^2-\big(1-g\big(f_{0,m}(0)\big)\big)^2\big| \mid L_n\geq 0\big] \ .\]
Note that in (\ref{eq0705}), the expectation in $\psi$ acts on the shifted environment $\theta_k\circ\Pi$. \\
As $n\rightarrow\infty$, $f_{0,n}(0)\rightarrow \mathbb{P}(Z_\infty=0\mid \Pi)$ a.s. Using Lemma \ref{lem25} then yields 
\[\lim_{n\rightarrow\infty}\psi(g,m,n)=\mathbf{E}^+\big[\big|\big(1-g\big(P_\infty\big)\big)^2- \big(1-g\big(f_{0,m}(0)\big)\big)^2\big|\big]  \ . \]
As $m\rightarrow\infty$, $f_{0,m}(0)\rightarrow P_\infty$ a.s. and thus by the dominated convergence theorem, the above term tends to zero as $m\rightarrow\infty$.
Applying this and $\mathbf{P}(L_{n-k}\geq 0)\sim\mathbf{P}(L_n\geq 0)$, we get that
\begin{align}
\sum_{k=0}^{l} &\big|\mathbf{E}[\mathbb{P}(Z_n>0\mid\Pi)^2-\mathbb{P}(Z_{\tau_n+m}>0\mid\Pi)^2;\tau_n=k] \big| \nonumber\\
&\leq \sum_{k=0}^{l } \mathbf{E}[\psi(f_{0,k},m,n-k); \tau_k=k]\ \mathbf{P}(L_{n}\geq 0) \leq \frac{\varepsilon}{2} \mathbf{P}(L_n\geq 0)\label{eq10}
\end{align}
if $m$ is large enough. This completes the proof.
\end{proof}
$\qquad$\textsl{Proof of Theorem \ref{theo2}.} The proof follows the same lines as the proof of \cite{agkv05}[Theorem 1.6]. 
Let $\phi$ be a bounded and continuous function on $D[0,1]$, the space of c\`adl\`ag functions on $[0,1]$. Recall that 
\[ S^n=\big(S^{n}_t\big)_{0\leq t\leq 1} = \tfrac{l(n)}{ n^{1/\alpha}} \big(S_{\lfloor nt\rfloor}\big)_{0\leq t\leq 1}\ .\] 
Then considering the change of measure to $\mathbf{P}$, (\ref{surv1}) 
and Theorem \ref{theo1}, as $n\rightarrow\infty$
\begin{align*}
\mathbb{E}[\phi(S^n)\mid Z_n=1]&=\frac{ \mathbf{E}\big[\phi(S^{n})\mathbb{P}(Z_n>0\mid\Pi)^2\big]}{\mathbf{E}\big[\mathbb{P}(Z_n>0\mid\Pi)^2\big]}\sim \frac 1\theta \frac{ \mathbf{E}\big[\phi(S^n)\mathbb{P}(Z_n>0\mid\Pi)^2\big]}{\mathbf{P}(L_n\geq 0)} \ .
\end{align*}
Thus it is enough to prove that as $n\rightarrow\infty$
\[\Big|\mathbf{E}\big[\phi(S^n)\mathbb{P}(Z_n>0\mid\Pi)^2\big]-\mathbf{E}\big[\phi(L^+)\big]\mathbf{E}\big[\mathbb{P}(Z_n>0\mid\Pi)^2\big]\Big|= o\big(\mathbf{P}(L_n\geq 0)\big) \ .\]
By the triangle inequality and Lemma \ref{lem12}, for every $\varepsilon>0$, if $m\in\mathbb{N}$ is large enough
\begin{align*}
\Big|\mathbf{E}&\big[\phi(S^n)\mathbb{P}(Z_n>0\mid\Pi)^2\big]-\mathbf{E}\big[\phi(L^+)\big]\mathbf{E}\big[\mathbb{P}(Z_n>0\mid\Pi)^2\big]\Big|\\
&\leq \Big|\mathbf{E}\big[\phi(S^n)\mathbb{P}(Z_{\tau_n+m}>0\mid\Pi)^2\big]-\mathbf{E}\big[\phi(L^+)\big]\mathbf{E}\big[\mathbb{P}(Z_{n}>0\mid\Pi)^2\big]\Big| \\
&\quad + \sup|\phi| \mathbf{E}\big[\big|\mathbb{P}(Z_n>0\mid\Pi)^2\big]-\mathbb{P}(Z_{\tau_n+m}>0\mid\Pi)^2\big|\big]\\
&=  \Big|\mathbf{E}\big[\phi(S^n)\mathbb{P}(Z_{\tau_n+m}>0\mid\Pi)^2\big]-\mathbf{E}\big[\phi(L^+)\big]\mathbf{E}\big[\mathbb{P}(Z_{n}>0\mid\Pi)^2\big]\Big| +\varepsilon \mathbf{P}(L_n\geq 0) \ .
\end{align*}
Next, we are going to prove that the first term can be bounded by $\varepsilon\mathbf{P}(L_n\geq 0)$ if $m$ is large enough. 
We decompose according to the time of the minimum. By (\ref{lem11}), for every $\varepsilon>0$ and for $l\in\mathbb{N}$ large enough, and as $\mathbb{P}(Z_n>0\mid\Pi)$ is 
non-increasing in $n$,
\begin{align*}
\Big|\mathbf{E}&\big[\phi(S^n)\mathbb{P}(Z_{\tau_n+m}>0\mid\Pi)^2\big]-\mathbf{E}\big[\phi(L^+)\big]\mathbf{E}\big[\mathbb{P}(Z_{n}>0\mid\Pi)^2\big]\Big|\\
&\leq\Big| \sum_{k=0}^l \Big(\mathbf{E}\big[\phi(S^n)\mathbb{P}(Z_{k+m}>0\mid\Pi)^2;\tau_n=k\big]-\mathbf{E}\big[\phi(L^+)\big]\mathbf{E}\big[\mathbb{P}(Z_n>0\mid\Pi)^2;\tau_n=k\big] \Big)\Big|\\
&\qquad +2\sup|\phi| \mathbf{E}\big[\mathbb{P}(Z_{k+m}>0\mid\Pi)^2; \tau_n>l\big] \\\
&\leq \Big| \sum_{k=0}^l \Big(\mathbf{E}\big[\phi(S^n)\mathbb{P}(Z_{k+m}>0\mid\Pi)^2;\tau_n=k\big]-\mathbf{E}\big[\phi(L^+)\big]\mathbf{E}\big[\mathbb{P}(Z_n>0\mid\Pi)^2;\tau_n=k\big] \Big)\Big|\\
&\quad +\varepsilon \mathbf{P}(L_n\geq 0) \ .
\end{align*}
Next we decompose the process $S^n$ according to generation $k$, i.e. let for $0\leq k\leq n$
\[S_t^{k,n} = n^{-1/\alpha} l(n) \ S_{\lfloor nt\rfloor\wedge k}  \]
and 
\[\overline{S}_t^{k,n} = n^{-1/\alpha} l(n) \ (S_{\lfloor nt\rfloor}-S_{\lfloor nt\rfloor\wedge k}) \ .  \]
Thus 
\[S^n= S^{k,n}+\overline{S}^{k,n} \ .\]
Recall the definition of $L_{k,n}=\min_{0\leq j\leq n-k} (S_{k+j}-S_k)$ from (\ref{min1}). Next, note that 
\begin{align}
  \sum_{k=0}^l  &\mathbf{E}\big[\phi(S^{k+m,n}+\overline{S}^{k+m,n}) \mathbb{P}(Z_{k+m}>0\mid\Pi)^2; \tau_n=k\big] \nonumber \\
&=  \sum_{k=0}^l  \mathbf{E}\Big[\mathbf{E}\big[\phi(S^{k+m,n}+\overline{S}^{k+m,n}) ; L_{k,n}\geq 0\mid \mathcal{F}_{k+m}\big] \mathbb{P}(Z_{k+m}>0\mid\Pi)^2; \tau_k=k\Big]\ .\label{eq1}
\end{align}
Conditioning on the environment yields
\begin{align}
\mathbf{E}&\Big[\mathbf{E}\big[\phi(S^{k+m,n}+\overline{S}^{k+m,n}) ; L_{k,n}\geq 0\mid \mathcal{F}_{k+m}\big] \mathbb{P}(Z_{k+m}>0|\Pi)\cdot I_{Z_{k+m}>0}; \tau_k=k\Big]\nonumber\\
&= \mathbf{E}\Big[\mathbf{E}\big[\phi(S^{k+m,n}+\overline{S}^{k+m,n}) ; L_{k,n}\geq 0\mid \mathcal{F}_{k+m}\big] \mathbb{P}(Z_{k+m}>0|\Pi)^2; \tau_k=k\Big]\ .\label{eq12}
\end{align}
Set for $w\in D[0,1]$ and $x\geq 0$
\[\psi(w,x):= \mathbf{E}\big[\phi(w+\overline{S}^{k+m,n});L_{k+m,n}\geq -x\big]\ .\]
Note that (see \cite{agkv05}[Proof of Theorem 1.5])
\[\{L_{k,n}\geq 0\}= \{L_{k,k+m}\geq 0\}\cap\{L_{k+m,n}\geq -(S_{k+m}-S_k)\}\ .\]
Thus, we may rewrite
\begin{align*}
\mathbf{E}&\big[\phi(S^{k+m,n}+\overline{S}^{k+m,n}); L_{k,n}\geq 0\mid \mathcal{F}_{k+m}\big]\\
&=\mathbf{E}\big[\phi(S^{k+m,n}+\overline{S}^{k+m,n}); L_{k,k+m}\geq 0,L_{k+m,n}\geq -(S_{k+m}-S_k)\mid \mathcal{F}_{k+m}\big] \\
&=\psi(S^{k+m,n}, S_{k+m}-S_k) \ I_{L_{k,k+m}\geq 0}\quad \text{a.s.}
\end{align*}
Using \cite{agkv05}[Lemma 2.3] and \cite{agkv05}[Lemma 2.1] yields (where the expectation in $\psi$ is taken with respect to the shifted environment $\theta_{k+m}\circ\Pi$)
\begin{align}
\psi(w,x) &= \mathbf{E}\big[\phi(w+\overline{S}^{k+m,n});L_{k+m,n}\geq -x\big]\nonumber\\
& = \mathbf{P}(L_{k+m,n}\geq -x) \big( \mathbf{E}\big[\phi(w+L^+)\big]+o(1)\big) \nonumber\\
&= u(x) \ \mathbf{P}(L_{n}\geq 0) \big( \mathbf{E}\big[\phi(w+L^+)\big]+o(1)\big) \ .\label{eq31}
\end{align}
Also note that for fixed $k$ and $m$, $S^{k+m,n}$ converges uniformly to 0 as $n\rightarrow\infty$ a.s. and that $\phi$ is continuous and bounded. Using this and (\ref{eq31}), 
we get that a.s.
\begin{align*}
 \mathbf{E}&\big[\phi(S^{k+m,n}+\overline{S}^{k+m,n}); L_{k,n}\geq 0\mid\mathcal{F}_{k+m}\big] = u(S_{k+m}-S_k) \mathbf{P}(L_n\geq 0)I_{L_{k,k+m}\geq 0} \big(\mathbf{E}[\phi(L^+)]+o(1)\big) \ .
\end{align*}
Inserting this into (\ref{eq1}) and using (\ref{eq12}) yields
\begin{align*}
 \sum_{k=0}^l  &\mathbf{E}\big[\phi(S^{k,n}+\overline{S}^{k+m,n}) \mathbb{P}(Z_{k+m}>0\mid\Pi)^2; \tau_n=k\big] \nonumber \\
&=  \big( \mathbf{E}\big[\phi(L^+)\big]+o(1)\big) \mathbf{P}(L_n\geq 0)\\
&\qquad \sum_{k=0}^l \mathbf{E}\Big[ u(S_{k+m}-S_k) \mathbb{P}(Z_{k+m}>0\mid\Pi)\cdot I_{Z_{k+m}>0}; \tau_k=k,L_{k,k+m}\geq 0\Big]\ .
\end{align*}
Finally, by the definition of $\mathbf{P}^+$, we have (recall $u(0)=1$)
\begin{align*}
\mathbf{E}&\Big[ u(S_{k+m}-S_k) \mathbb{P}(Z_{k+m}>0\mid\Pi)\cdot I_{Z_{k+m}>0}; \tau_k=k,L_{k,k+m}\geq 0\Big]\\
&= \mathbf{E}\Big[\mathbf{E}^+\Big[\mathbb{P}(Z_{k+m}>0\mid\Pi)\cdot I_{Z_{k+m}>0}\mid\mathcal{F}_k\Big];\tau_k=k\Big]\\
&\stackrel{m\rightarrow\infty}{\longrightarrow} \mathbf{E}\big[ V_{\infty} (Z_k, f_{0,k}); \tau_k=k\big] \ ,
\end{align*}
where $V_{\infty} (Z_k, f_{0,k})$ is defined in Lemma \ref{lem10}. Inserting all this yields
\begin{align*}
 \sum_{k=0}^l \mathbf{E}&\big[\phi(S^n) \mathbb{P}(Z_{k+m}>0\mid\Pi)^2;\tau_n=k\big] \\
&=
\big(\mathbf{E}\big[\phi(L^+)\big]+o(1)\big) \mathbf{P}(L_n\geq 0) \sum_{k=0}^l \Big(\mathbf{E}\big[ V_{\infty} (Z_k, f_{0,k}); \tau_k=k\big]+o(1)\Big)
\end{align*}
On the other hand, by Lemmas \ref{lem10} and \ref{lem12},
\begin{align*}
 \mathbf{E}\big[\mathbb{P}(Z_n>0\mid\Pi)^2;\tau_n=k\big]= \mathbf{P}(L_n\geq 0) \sum_{k=0}^l \Big(\mathbf{E}\big[ V_{\infty} (Z_k, f_{0,k}); \tau_k=k\big]+o(1)\Big) \ .
\end{align*}
Thus for every $\varepsilon>0$ if $l$ and $m$ are large enough,
\begin{align*}
 \Big| \sum_{k=0}^l \mathbf{E}\big[\phi(S^n)\mathbb{P}(Z_{k+m}>0\mid\Pi)^2;\tau_n=k\big]-\mathbf{E}\big[\phi(L^+)\big]\mathbf{E}\big[\mathbb{P}(Z_n>0\mid\Pi)^2;\tau_n=k\big] \Big|\leq \varepsilon \mathbf{P}(L_n\geq 0) \ .
\end{align*}
This proves the theorem.\qed \medskip\\

\subsection{Proof of Theorem \ref{theo23} and \ref{theo3}}
The following lemma describes the probability that the process has some value, conditioned on a favorable environment.
\begin{lemma}\label{le_35}
For every $z,k\in\mathbb{N}$,
\[\lim_{n\rightarrow\infty} \mathbf{E}\big[e^{S_n} \mathbb{P}_z(Z_n=k\mid\Pi) \mid L_n\geq 0\big]=z\  
\mathbf{E}^+\big[\mathbb{P}(Z_\infty>0\mid \Pi)^2\mathbb{P}(Z_\infty=0\mid \Pi)^{z-1}\big] \ ,\]
where the limit does not depend on $k$. 
\end{lemma}
\begin{proof}
We will prove the lemma by induction with respect to $z$.\\
 For $z=1$,  the explicit formula for the probability (\ref{geo2}) yields 
\begin{align*}
 \lim_{n\rightarrow\infty} &\mathbf{E}\big[e^{S_n} \mathbb{P}(Z_n=k\mid\Pi) \mid L_n\geq 0\big] = \lim_{n\rightarrow\infty} \mathbf{E}\big[\mathbb{P}(Z_n>0\mid \Pi)^2 H_n^{k-1} \mid L_n\geq 0\big]\ .
\end{align*}
Recall that 
\[H_n = \frac{\sum_{k=0}^{n-1} \eta_{k+1} e^{-S_k}}{e^{-S_n} +\sum_{k=0}^{n-1} \eta_{k+1} e^{-S_k}}\ .\]
Under $\mathbf{P}^+$, $S_n\rightarrow\infty$ a.s. and thus $e^{-S_n}\rightarrow 0$ a.s. Consequently $H_n\rightarrow 1$ $\mathbf{P}^+$-a.s. and as $n\rightarrow\infty$
\[e^{S_n}\mathbb{P}(Z_n=k\mid\Pi) \rightarrow \mathbb{P}(Z_\infty>0\mid\Pi)^2 \quad \mathbf{P}^+-\text{a.s.} \]
Using this together with Lemma \ref{lem25} and $H_n\leq 1$ yields for every $k\in\mathbb{N}$
\begin{align*}
 \lim_{n\rightarrow\infty} &\mathbf{E}\big[e^{S_n} \mathbb{P}(Z_n=k\mid\Pi) \mid L_n\geq 0\big]= \mathbf{E}^+ \big[\mathbb{P}(Z_\infty>0\mid \Pi)^2 \big]\ ,
\end{align*}
which completes the proof for $z=1$.\\
Let us now assume that for $z\in\mathbb{N}$
\[e^{S_n}\mathbb{P}_z(Z_n=k\mid\Pi) \rightarrow z\ \mathbb{P}(Z_\infty>0\mid\Pi)^2 \mathbb{P}(Z_\infty=0\mid \Pi)^{z-1}\quad \mathbf{P}^+-\text{a.s.} \]
and thus
\[ \lim_{n\rightarrow\infty} \mathbf{E}\big[e^{S_n} \mathbb{P}_z(Z_n=k\mid\Pi) \mid L_n\geq 0\big] =z\  \mathbf{E}^+\big[\mathbb{P}(Z_\infty>0\mid \Pi)^2
\mathbb{P}(Z_\infty=0\mid \Pi)^{z-1}\big]\ .\]
Then starting from $z+1$-many individuals, $Z_n$ is the sum of $z+1$-many independent and identically distributed random variables. Thus a.s.
\begin{align*}
  e^{S_n}&\mathbb{P}_{z+1}(Z_n=k\mid \Pi) =e^{S_n} \sum_{j=0}^{k} \mathbb{P}(Z_n=j\mid\Pi) \mathbb{P}_{z}(Z_n=k-j\mid \Pi)\\
&= \mathbb{P}(Z_n=0\mid\Pi)  e^{S_n}\mathbb{P}_{z}(Z_n=k\mid \Pi)+\sum_{j=1}^{k-1} \mathbb{P}(Z_n=j\mid\Pi) e^{S_n}\mathbb{P}_{z}(Z_n=k-j\mid \Pi)\\
&+ e^{S_n}\mathbb{P}(Z_n=k\mid\Pi)  \mathbb{P}_{z}(Z_n=0\mid \Pi)\ .
\end{align*}
For the first summand, by assumption of the induction 
\[e^{S_n}\mathbb{P}_{z}(Z_n=k\mid \Pi)\stackrel{n\rightarrow\infty}{\longrightarrow} z\ \mathbb{P}(Z_\infty>0\mid \Pi)^2\mathbb{P}(Z_\infty=0)^{z-1}\]
$\mathbf{P}^+$-a.s. for every $k>0$. Under $\mathbf{P}^+$, $ \mathbb{P}(Z_n=0\mid\Pi)\rightarrow\mathbb{P}(Z_\infty=0\mid\Pi)$. Thus as $n\rightarrow\infty$, $\mathbf{P}^+$ a.s.  
\[\mathbb{P}(Z_n=0\mid\Pi)  e^{S_n}\mathbb{P}_{z}(Z_n=k\mid \Pi)\rightarrow z\ \mathbb{P}(Z_\infty>0\mid \Pi)^2\mathbb{P}(Z_\infty=0\mid \Pi)^{z}\ . \]
As to the second part, again $e^{S_n}\mathbb{P}_{z}(Z_n=k-j\mid \Pi)$ converges $\mathbf{P}^+$ a.s. for every $k-j>0$. Note that $ e^{S_n}\mathbb{P}(Z_n=j\mid\Pi)$ 
converges $\mathbf{P}^+$-a.s. and thus, as $S_n\rightarrow\infty$ $\mathbf{P}^+$-a.s., we have $\mathbb{P}(Z_n=j\mid\Pi)\rightarrow0$ $\mathbf{P}^+$ a.s. Consequently, 
as there are only finitely many summands, 
\[\sum_{j=1}^{k-1} \mathbb{P}(Z_n=j\mid\Pi) e^{S_n}\mathbb{P}_{z}(Z_n=k-j\mid \Pi)\stackrel{n\rightarrow\infty}{\longrightarrow} 0\quad \mathbf{P}^+-\text{a.s.}\]
For the last summand, note that
\[ \mathbb{P}_{z}(Z_n=0\mid \Pi)= \mathbb{P}(Z_n=0\mid \Pi)^z\stackrel{n\rightarrow\infty}{\longrightarrow} \mathbb{P}(Z_\infty=0\mid \Pi)^z\]
and consequently, as $n\rightarrow\infty$, $\mathbf{P}^+$-a.s.
\[e^{S_n}\mathbb{P}(Z_n=k\mid\Pi)  \mathbb{P}_{z}(Z_n=0\mid \Pi)\rightarrow \mathbb{P}(Z_\infty>0\mid \Pi)^2 \mathbb{P}(Z_n=0\mid \Pi)^z\ .\]
Putting this together and applying Lemma \ref{lem25}, we get that
\begin{align*}
  \lim_{n\rightarrow\infty} &\mathbf{E}\big[e^{S_n} \mathbb{P}_{z+1}(Z_n=k\mid\Pi) \mid L_n\geq 0\big]\\
&=z\ \mathbf{E}^+\big[\mathbb{P}(Z_\infty>0\mid \Pi)^2\mathbb{P}(Z_\infty=0\mid \Pi)^{z}\big]+\mathbf{E}^+\big[\mathbb{P}(Z_\infty>0\mid \Pi)^2\mathbb{P}(Z_\infty=0\mid\Pi)^{z}\big]  \\
&= (z+1)\ \mathbf{E}^+\big[\mathbb{P}(Z_\infty>0\mid \Pi)^2\mathbb{P}(Z_\infty=0\mid\Pi)^{z}\big] \ .
\end{align*}
This ends up the induction
\end{proof}
$\qquad$\textsl{Proof of Theorem \ref{theo23}.}
Fix $c\in\mathbb{N}$. Then for $1\leq k\leq c$
\begin{align*}
 \mathbb{P}(Z_n=k\mid 1\leq Z_n\leq c)= \frac{ \mathbb{P}(Z_n=k)}{\sum_{j=1}^c  \mathbb{P}(Z_n=j)} \ .
ends\end{align*}
Next, using the change of measure and the decomposition according to the global minimum, for every $0\leq m\leq n$,
\begin{align*}
 \mathbb{P}(Z_n=k) &= \gamma^n\sum_{i=0}^m \mathbf{E}\big[e^{S_n}\mathbb{P}(Z_n=k\mid\Pi) ; \tau_i=i, L_{i,n}\geq 0\big]\\
&+\mathbb{P}(Z_n=k,\tau_n>m) \ .
\end{align*}
Let $\varepsilon>0$. By Lemma \ref{lemhelp} and Theorem \ref{theo1}, the second term is can be bounded by $\varepsilon\mathbb{P}(Z_n=1)$ for $m$ 
large enough and as $n\rightarrow\infty$. Examining the first term, we get that
\begin{align*}
 \gamma^n\sum_{i=0}^m& \mathbf{E}\big[e^{S_n}\mathbb{P}(Z_n=k\mid\Pi) ; \tau_i=i, L_{i,n}\geq 0\big] = \gamma^n \sum_{i=0}^m 
\mathbf{E}\big[e^{S_i}e^{S_n-S_i}\mathbb{P}(Z_n=k\mid\Pi) ; \tau_i=i, L_{i,n}\geq 0\big]\\
&=\gamma^n \sum_{i=0}^m 
\mathbf{E}\big[e^{S_i} \psi(Z_i,n-i); \tau_i=i\big] \mathbf{P}(L_{n-i}\geq 0) \ ,
\end{align*}
where
\begin{align*}
 \psi(z,n)&=\mathbf{E}\big[e^{S_n} \mathbb{P}_z(Z_n=k\mid \Pi) \mid L_{n}\geq 0\big] \ .
\end{align*}
The expectation in $\psi$ is taken with respect to the shifted environment $\theta_i\circ\Pi$. Note that as $n\rightarrow\infty$, using Lemma \ref{le_35},
\begin{align*}
 \psi(z,n)\rightarrow z\ \mathbf{E}^+\big[\mathbb{P}(Z_\infty>0\mid \Pi)^2 \mathbb{P}(Z_\infty=0\mid \Pi)^{z-1}\big] \ .
\end{align*}
This term does not depend on $k$ as $n\rightarrow\infty$. Thus for every $k\in\mathbb{N}$, 
\begin{align*}
 \lim_{n\rightarrow\infty}\mathbb{P}(Z_n=1)/\lim_{n\rightarrow\infty}\mathbb{P}(Z_n=k)= 1\ .
\end{align*}
Theorem \ref{theo23} immediately results from this. \qed\medskip\\
Our next result concerns the time of the prospective global minima. For convenience, 
we shorten $\tau_{\lfloor nt\rfloor}$, $Z_{\lfloor nt\rfloor}$, $S_{\lfloor nt\rfloor}$ to  $\tau_{nt}$, $Z_{nt}$, $S_{nt}$, i.e. we drop $\lfloor \cdot \rfloor$ in the indices.
\begin{lemma}\label{le_39}
 For every $\varepsilon>0$ and $t\in(0,1)$, there is an $m\in\mathbb{N}$ such that 
\[\mathbf{P}(\tau_{nt,n}>\lfloor nt\rfloor +m \mid L_n\geq 0)\leq\varepsilon \ .\]
\end{lemma}
\begin{proof}
The main idea is to apply \cite{agkv05}[Lemma 2.2]. First note that decomposing at time $\lfloor nt\rfloor$ and by independence,
\begin{align*}
\mathbf{P}(\tau_{nt,n}>\lfloor nt\rfloor +m, L_n\geq 0) &= \int_{0}^\infty \mathbf{P}(\tau_{n-nt}>m, L_{n-nt}\geq -x)\mathbf{P}(S_{nt}\in dx, L_{nt}\geq 0)\ dx\ . 
\end{align*} 
Next, we can rewrite 
\begin{align*}
\mathbf{P}(\tau_{n-nt}>m, L_{n-nt}\geq -x)&=\sum_{k=m+1}^{n-\lfloor nt\rfloor} \mathbf{P}(\tau_k=k,S_k\geq -x) \mathbf{P}(L_{n-nt-k}\geq 0) \\
&= \sum_{k=m+1}^{n-\lfloor nt\rfloor} \mathbf{E}\big[u(-S_k);\tau_k=k\big] \mathbf{P}(L_{n-nt-k}\geq 0) ,
\end{align*}
where $u(y):=1_{y\leq x}$. Obviously, $u$ is nonnegative, nonincreasing in $y$ and $\int_{0}^\infty u(y)dy=x<\infty$. Thus all conditions of \cite{agkv05}[Lemma 2.2] are met. Applying this lemma yields for every $x$, every $\varepsilon>0$ and if $m$ is large enough
\begin{align*}
 \sum_{k=m+1}^{n-\lfloor nt\rfloor} \mathbf{E}\big[u(-S_k);\tau_k=k\big] \mathbf{P}(L_{n-nt-k}\geq 0) \leq \varepsilon \mathbf{P}(L_{n-nt}\geq 0) \ .
\end{align*}
Thus we get that
\begin{align*}
\mathbf{P}(\tau_{nt,n}>\lfloor nt\rfloor +m, L_n\geq 0)&\leq \varepsilon \int_0^\infty \mathbf{P}(S_{nt}\in dx, L_{nt}\geq 0) \mathbf{P}(L_{n-nt}\geq 0) \ dx\\
&= \varepsilon  \mathbf{P}(L_{nt}\geq 0)  \mathbf{P}(L_{n-nt}\geq 0) \leq \varepsilon \mathbf{P}(L_n\geq 0) \ . 
\end{align*}
For the last step, note that $\{L_{nt}\geq 0\}\cap\{L_{nt,n}\geq 0\}$ implies $\{L_n\geq 0\}$. Thus, as the two terms in the first event are independent, we get that
\[\mathbf{P}(L_{nt}\geq 0)  \mathbf{P}(L_{n-nt}\geq 0) \leq \mathbf{P}(L_n\geq 0)\ .\] 
\end{proof}
$\qquad$\textsl{Proof of Theorem \ref{theo3}.} Let $0\leq m\leq n$ and $t\in (0,1)$. Decomposing according to the global minimum yields
\begin{align}
 \mathbb{P}\big(Z_{\tau_{nt,n}}=z,Z_n=1\big) &= \sum_{k=0}^m \mathbb{P}\big(Z_{\tau_{nt,n}}=z,Z_n=1, \tau_n=k\big) \nonumber\\
&\qquad+
 \mathbb{P}\big(Z_{\tau_{nt,n}}=z,Z_n=1, \tau_n>m\big)\ ,\label{eq1425}
\end{align}
By Lemma \ref{lemhelp} and Theorem \ref{theo1}, for every $\varepsilon>0$ and $m$ large enough,
\begin{align*}
 \limsup_{n\rightarrow\infty} \mathbb{P}(Z_n=1)^{-1} \mathbb{P}\big(Z_{\tau_{nt,n}}=z,Z_n=1, \tau_n>m\big)\leq \varepsilon \ .
\end{align*}
Thus we only have to consider the first sum on the right-hand side in (\ref{eq1425}). Let $m<\ell< n-\lfloor nt\rfloor$ be specified later. Then we get that
\begin{align}
 &\sum_{k=0}^m \mathbb{P}\big(Z_{\tau_{nt,n}}=z,Z_n=1, \tau_n=k\big)\nonumber \\
&= \gamma^n \sum_{k=0}^m \sum_{j=0}^{n-\lfloor nt\rfloor}\mathbf{E}\big[e^{S_n}\mathbb{P}(Z_{nt+j}=z,Z_n=1\mid \Pi);  \tau_k=k, L_{k,n}\geq 0,\tau_{nt,n}=\lfloor nt\rfloor+j\big] \nonumber\\
&= \gamma^n \sum_{k=0}^m \sum_{j=0}^{\ell}\mathbf{E}\big[e^{S_{nt+j}} e^{S_n-S_{nt+j}}\mathbb{P}(Z_{nt+j}=z,Z_n=1\mid \Pi);  \tau_k=k, L_{k,n}\geq 0,\tau_{nt,n}=\lfloor nt\rfloor+j\big] \nonumber\\
&+\gamma^n \sum_{k=0}^m \sum_{j=\ell+1}^{n-\lfloor nt\rfloor}\mathbf{E}\big[e^{S_n}\mathbb{P}(Z_{nt+j}=z,Z_n=1\mid \Pi);  \tau_k=k, L_{k,n}\geq 0,\tau_{nt,n}=\lfloor nt\rfloor+j\big]\nonumber\\
&=:s_1+s_2\ . \label{eqbuda1}
\end{align}
In view of $\mathbb{P}(Z_{nt+j}=z,Z_n=1\mid \Pi)\leq \mathbb{P}(Z_n=1\mid \Pi)\leq e^{-S_n}$ a.s. (see this special property of generalized geometric offspring in (\ref{geo2})), the second summand can be bounded by
\begin{align*}
 s_2&\leq \gamma^n \sum_{k=0}^m \sum_{j=\ell+1}^{n-\lfloor nt\rfloor}\mathbf{P}\big(\tau_k=k, L_{k,n}\geq 0,\tau_{nt,n}=\lfloor nt\rfloor+j\big)\\
&=\gamma^n \sum_{k=0}^m \mathbf{P}\big(\tau_k=k)\mathbf{P}(L_{n-k}\geq 0) \mathbf{P}(\tau_{nt-k,n-k}>\lfloor nt\rfloor+\ell-k\mid L_{n-k}\geq 0\big)\ .
\end{align*}
Using Lemma \ref{le_39} yields for every fixed $m\in\mathbb{N}$ and $\varepsilon>0$, if $\ell=\ell(m,\varepsilon)$ is large enough,
\begin{align*}
s_2 &\leq \frac{\varepsilon}{\sum_{k=0}^m \mathbf{P}\big(\tau_k=k)} \gamma^n  \sum_{k=0}^m \mathbf{P}\big(\tau_k=k)\mathbf{P}(L_{n-k}\geq 0)\\
&\leq \frac{\varepsilon}{\sum_{k=0}^m \mathbf{P}\big(\tau_k=k)} \gamma^n \mathbf{P}(L_{n-m}\geq 0) \sum_{k=0}^m \mathbf{P}(\tau_{k}=k)\\
&\leq \varepsilon \ \gamma^n \mathbf{P}(L_{n-m}\geq 0)\ .
\end{align*}
In the second step, we have used the fact that $\mathbf{P}(L_n\geq 0)$ is decreasing in $n$. \\
As $\mathbf{P}(L_{n-m}\geq 0)\sim \mathbf{P}(L_n\geq 0)$ as $n\rightarrow\infty$ and using Theorem \ref{theo1}, we get that
\[s_2\leq \varepsilon \mathbb{P}(Z_n=1)\ .\]
 Thus $s_2$ may be neglected as $n\rightarrow\infty$ and then $m\rightarrow\infty$.\\
Let us turn to the term $s_1$. First note that the event $\{\tau_{nt,n}=j\}$ can be written as
\[\{\tau_{nt,n}=j\}= \{ \tau_{nt,j}=j \}\cap\{L_{j,n}\geq 0\}\ .\]
Moreover,
\[\{\tau_{nt,n}=j, L_n\geq 0\}= \{ \tau_{nt,j}=j, L_j\geq 0 \}\cap\{L_{j,n}\geq 0\}\ ,\]
where both events are independent. Conditioning on $\mathcal{F}_{nt+j}$, we get that
\begin{align}
&s_1= \gamma^n \sum_{k=0}^m \sum_{j=0}^{\ell}\mathbf{E}\big[e^{S_{nt+j}}\mathbb{P}(Z_{nt+j}=z\mid \Pi);  \tau_k=k, L_{k,nt+j}\geq 0,\tau_{nt,nt+j}=\lfloor nt\rfloor+j\big]\nonumber\\
& \qquad  \qquad \cdot\mathbf{E}\big[e^{S_{n-nt-j}} \mathbb{P}_z(Z_{n-nt-j}=1\mid \Pi);L_{n-nt-j}\geq 0 \big] \nonumber\\
&= \gamma^n \sum_{k=0}^m \sum_{j=0}^{\ell}\mathbf{E}\big[e^{S_{nt+j}}\mathbb{P}(Z_{nt+j}=z\mid \Pi);  \tau_k=k, L_{k,nt+j}\geq 0,\tau_{nt,nt+j}=\lfloor nt\rfloor+j\big]\nonumber\\
& \qquad  \qquad \cdot\mathbf{E}\big[e^{S_{n-nt-j}} \mathbb{P}_z(Z_{n-nt-j}=1\mid \Pi)|L_{n-nt-j}\geq 0 \big] \mathbf{P}(L_{n-nt-j}\geq 0 )\ . \label{eq955}
\end{align}
Next, we use the explicit formula for the probability in (\ref{geo2}) and get that a.s.
\begin{align*}
 e^{S_{nt+j}}\mathbb{P}(Z_{nt+j}=z\mid \Pi)= \mathbb{P}(Z_{nt+j}>0\mid \Pi)^2 H_{nt+j}^{z-1}\ . 
\end{align*}
Using Lemma \ref{le_35} for $k=1$, as $n\rightarrow\infty$
\begin{align*}
 \mathbf{E}\big[e^{S_{n-nt-j}} &\mathbb{P}_z(Z_{n-nt-j}=1\mid \Pi)\mid L_{n-nt-j}\geq 0 \big] \\
& =z\ \mathbf{E}^+\big[\mathbb{P}(Z_\infty>0\mid\Pi)^2 \mathbb{P}(Z_\infty=0\mid\Pi)^{z-1} \big]+o(1)\ .
\end{align*}
Inserting this into (\ref{eq955}) yields 
\begin{align*}
 s_1&= \gamma^n \sum_{k=0}^m \sum_{j=0}^{\ell}\mathbf{E}\big[\mathbb{P}(Z_{nt+j}>0\mid \Pi)^2 H_{nt+j}^{z-1};  \tau_k=k, L_{k,nt+j}\geq 0,\tau_{nt,nt+j}=\lfloor nt\rfloor+j\big]\nonumber \\
&\qquad  \qquad \cdot\Big(z\ \mathbf{E}^+\big[\mathbb{P}(Z_\infty>0\mid\Pi)^2 \mathbb{P}(Z_\infty=0\mid\Pi)^{z-1} \big]+o(1)\Big)\mathbf{P}(L_{n-nt-j}\geq 0 )\nonumber \\
&= \gamma^n \sum_{k=0}^m \sum_{j=0}^{\ell}\mathbf{E}\big[\mathbb{P}(Z_{nt+j}>0\mid \Pi)^2 H_{nt+j}^{z-1};  \tau_k=k, L_{k,n}\geq 0,\tau_{nt,n}=\lfloor nt\rfloor+j\big]\nonumber \\
&\qquad  \qquad \cdot\Big(z\ \mathbf{E}^+\big[\mathbb{P}(Z_\infty>0\mid\Pi)^2 \mathbb{P}(Z_\infty=0\mid\Pi)^{z-1} \big]+o(1)\Big)\nonumber \\
&= \gamma^n \sum_{k=0}^m \mathbf{E}\big[\mathbb{P}(Z_{\tau_{nt,n}}>0\mid \Pi)^2 H_{\tau_{nt,n}}^{z-1};  \tau_k=k, L_{k,n}\geq 0,\tau_{nt,n}\leq \lfloor nt\rfloor +\ell \big]\nonumber \\
&\qquad  \qquad \cdot\Big(z\ \mathbf{E}^+\big[\mathbb{P}(Z_\infty>0\mid\Pi)^2 \mathbb{P}(Z_\infty=0\mid\Pi)^{z-1} \big]+o(1)\Big)\ .\nonumber 
\end{align*}
Note that always $\tau_{nt,n} \geq \tau_n$. In the first factor, we can condition on $\{L_{k,n}\geq 0\}$ and get that
\begin{align*}
 &\mathbf{E}\big[\mathbb{P}(Z_{\tau_{nt,n}}>0\mid \Pi)^2 H_{\tau_{nt,n}}^{z-1};  \tau_k=k, L_{k,n}\geq 0,\tau_{nt,n}\leq \lfloor nt\rfloor+\ell \big] \\
&=\mathbf{E}\Big[\mathbf{E}\big[\mathbb{P}(Z_{\tau_{nt,n}}>0\mid \Pi)^2 H_{\tau_{nt,n}}^{z-1},\tau_{nt,n}\leq \lfloor nt\rfloor+\ell \mid L_{k,n}\geq 0, \mathcal{F}_k\big];  
\tau_k=k\Big]\mathbf{P}(L_{n-k}\geq 0)\\
&\sim \mathbf{P}(L_n\geq 0)\mathbf{E}\Big[\mathbf{E}\big[\mathbb{P}(Z_{\tau_{nt,n}}>0\mid \Pi)^2 H_{\tau_{nt,n}}^{z-1};\tau_{nt,n}\leq \lfloor nt\rfloor+\ell \mid L_{k,n}\geq 0,\mathcal{F}_k\big];  \tau_k=k\Big] \ .
\end{align*}
By \cite{agkv05}[Proof of Lemma 2.7], $\tau_{nt,n}\rightarrow \infty$ $\mathbf{P}^+$-a.s. as $n\rightarrow\infty$.
Thus, $\mathbb{P}(Z_{\tau_{nt,n}}>0\mid \Pi)$ converges a.s. with respect to $\mathbf{P}^+$.  Moreover, 
as under $\mathbf{P}^+$, $S_n\rightarrow\infty$ and $S_{\tau_{nt,n}}\rightarrow\infty$ as $n\rightarrow \infty$, we get $H_{\tau_{nt,n}}^{z-1}\rightarrow 1$ $\mathbf{P}^+$-a.s.
Using this, Lemma \ref{lem25} and Lemma \ref{le_39}, we get a.s. as $n\rightarrow\infty$
\[ \mathbf{E}\big[\mathbb{P}(Z_{\tau_{nt,n}}>0\mid \Pi)^2 H_{\tau_{nt,n}}^{z-1};\tau_{nt,n}\leq \lfloor nt\rfloor+\ell \mid L_{k,n}\geq 0,\mathcal{F}_k\big] \rightarrow 
\mathbf{E}^+ \big[(1-f_{0,k}\big(P^k_\infty)\big)^2\big]\ , \]
where $P^k_\infty$ has been defined in (\ref{eq_pk}) and the expectation only acts on the shifted environment $\theta_k\circ\Pi$.
We can now formulate the limit. As $s_2$ and thus the corresponding probability on the event $\{\tau_{nt,n}> nt+m \}$ 
can be neglected as $n\rightarrow\infty$ and if $m\rightarrow\infty$, we get that
\begin{align*}
 \mathbb{P}&(Z_{\tau_{nt,n}}=z, Z_n=1) \\
&\sim z\ \mathbf{E}^+\big[\mathbb{P}(Z_\infty>0\mid\Pi)^2 \mathbb{P}(Z_\infty=0\mid\Pi)^{z-1} \big]\gamma^n \mathbf{P}(L_n\geq 0) \sum_{k=0}^\infty \mathbf{E}\big[\mathbf{E}^+\big[(1-f_{0,k}\big(P^k_\infty)\big)^2\big]; \tau_k=k\big]\nonumber\ .
\end{align*}
Together with Theorem \ref{theo1} and the formula for $\theta$ in (\ref{reptheta2}), we get that as $n\rightarrow\infty$ 
\begin{align*}
 \lim_{n\rightarrow\infty}& \mathbb{P}\big(Z_{\tau_{nt,n}}=z\mid Z_n=1\big)=z\ \mathbf{E}^+\big[\mathbb{P}(Z_\infty>0\mid\Pi)^2 \mathbb{P}(Z_\infty=0\mid\Pi)^{z-1} \big]  \ .
\end{align*}
As proved in Theorem \ref{thstrong2}, this is indeed a probability distribution on $\mathbb{N}$.
\qed \medskip\\
\textbf{Acknowledgment.}  The author thanks Vincent Bansaye from Ecole Polytechnique, Palaiseau for fruitful discussions on this topic. The author also would like to thank the anonymous referees for their useful comments and improvements.

\end{document}